\documentclass[11pt,leqno]{article}

\pdfoutput=1
\usepackage[a4paper,top=1.2in,bottom=1in,left=1.2in,right=1.2in]{geometry}

\usepackage{caption}
\usepackage{titlesec}
\renewcommand{\thesubsection}{\thesection.\arabic{subsection}}
\renewcommand{\thesubsubsection}{\thesubsection.\Alph{subsubsection}}
\titleformat{\section}[hang]{\normalfont\bfseries}{\thesection}{.5em}{}
\titlelabel{\subsubsection}{\empty}
\titleformat{\subsection}{\normalfont\bfseries}{\thesubsection.}{.5em}{}[]
\titleformat{\subsubsection}[runin]{\normalfont}{\thesubsubsection.}{0em}{}[]

\renewcommand{\abstract}[1]{{\gdef\thepoabstract{#1}}}

\usepackage{amsthm}
\usepackage{thmtools}

\makeatletter
\renewcommand\maketitle
{
\thispagestyle{empty}
\rmfamily\selectfont
\ifdefined\@title{\noindent\Large\bfseries\centering \@title \par}\else\fi
\vspace{2em}
\ifdefined\@author{\centering\normalfont \@author \par}
\vspace{.5em}
\ifdefined\thepoaffiliation{\noindent\small\thepoaffiliation\par}\fi\vspace{3em}\else\fi
\ifdefined\thepoabstract{\small\noindent{{\bfseries Abstract}.\;\;}\thepoabstract\par\vspace{2em}}\else\fi
\ifdefined\thepokeywords{{\noindent\bfseries Keywords\;\;}\thepokeywords\par\vspace{5em}}\else\fi
\ifdefined\theporuntitle{\fancyhead[L]{\footnotesize\theporuntitle}}\fi
}
\makeatother

\usepackage[usenames,dvipsnames]{color}
\definecolor{RefColor}{rgb}{0,0,.85}
\definecolor{UrlColor}{rgb}{.5,.5,.5}%
\RequirePackage[colorlinks,linkcolor=RefColor,citecolor=RefColor,urlcolor=UrlColor,linktoc=page]{hyperref}
\hypersetup{pdfinfo={Subject={ }}}
\RequirePackage[OT1]{fontenc}
\usepackage[numbers,sort]{natbib}

\usepackage[english]{babel}

\usepackage{enumitem}
\setlist[itemize]{leftmargin=1.5em}

\usepackage{amsmath,amssymb,amscd,amsfonts,amsthm,mathtools}
\usepackage[noabbrev,capitalize]{cleveref}
\usepackage{dsfont}
\usepackage{thmtools}
\usepackage{nccmath}
\usepackage{scalerel}

\usepackage{bibentry}

\usepackage{booktabs}
\usepackage{tikz}
\usepackage[low-sup]{subdepth}
\usetikzlibrary{calc}
\usetikzlibrary{decorations.markings,decorations.pathreplacing}
\tikzstyle{mybraces}=[mirrorbrace/.style={
          decoration={brace, mirror},
          decorate},brace/.style={
          decoration={brace},
          decorate}]
\usepackage{pst-node}
\usepackage{tikz-cd}

\usepackage{tocloft}

\cftsetpnumwidth{5cm}
\cftsetrmarg{6cm}

\usepackage{wrapfig}
\usepackage{subcaption}
\usepackage[nottoc]{tocbibind}
\settocbibname{References}

\theoremstyle{plain}
\declaretheoremstyle[postheadspace=.4em,headfont=\bfseries,bodyfont=\itshape,spaceabove=8pt,
spacebelow=10pt]{basic}
\theoremstyle{basic}
\declaretheorem[style=basic,name={Theorem}]{theorem}
\declaretheorem[style=basic,sibling=theorem,name={Lemma}]{lemma}
\declaretheorem[style=basic,sibling=theorem,name={Fact}]{fact}
\declaretheorem[style=basic,sibling=theorem,name={Proposition}]{proposition}
\declaretheorem[style=basic,sibling=theorem,name={Corollary}]{corollary}
\theoremstyle{definition}

\declaretheorem[style=definition,name={Remark}]{remark}
\declaretheorem[style=definition,name={Remark},numbered=no]{remark*}

\usepackage{times}

\newcommand{\msup}{\sup\nolimits}

\def\bx{\mathbf{x}}

\def\N{\mathbb{N}}

\def\P{\mathbb{P}}

\def\R{\mathbb{R}}

\def\cN{\mathcal{N}}

\newcommand{\mean}{\mathbb{E}}
\newcommand{\Var}{\text{\rm Var}}

\DeclareMathOperator{\msum}{\medmath\sum}

\DeclareMathOperator{\mint}{\scaleobj{.8}{\int}}

\begin{document}

\title{Slow rates of approximation of U-statistics and V-statistics\\ by quadratic forms of Gaussians}
\author{Kevin Han Huang and Peter Orbanz\\[1em]\footnotesize Gatsby Unit, University College London}

\begin{abstract}
  {
    We construct examples of degree-two U- and V-statistics of $n$ i.i.d.~heavy-tailed random vectors in $\R^{d(n)}$, whose $\nu$-th moments exist for ${\nu > 2}$, and provide tight bounds on the error of approximating both statistics by a quadratic form of Gaussians. In the case ${\nu=3}$, the error of approximation is $\Theta(n^{-1/12})$. The proof adapts a result of Huang, Austern and Orbanz \cite{huang2024gaussian} to U- and V-statistics. The lower bound for U-statistics is  a simple example of the concept of variance domination used in \cite{huang2024gaussian}.
  }
\end{abstract}

\maketitle

\section{Introduction} We consider distributional approximations of degree-two U- and V-statistics that are of the form
\begin{align*}
    u_n(x_1, \ldots, x_n) \;\coloneqq&\; \mfrac{\sum_{1 \leq i \neq j \leq n} k_u(x_i, x_j)}{n(n-1)} 
    &\text{ and }&&
    v_n(x_1, \ldots, x_n) \;\coloneqq&\; \mfrac{\sum_{1 \leq i,j \leq n} k_v(x_i, x_j)}{n^2} \;.
\end{align*}
Here, $x_1, \ldots, x_n \in \R^{d(n)}$ are elements of a high-dimensional Euclidean space, i.e.~the dimension $d(n)$ may depend on $n$, and $k_u: \R^{d(n)} \times \R^{d(n)} \rightarrow \R$ and $k_v: \R^{d(n)} \times \R^{d(n)} \rightarrow \R$ are symmetric functions. These statistics arise frequently in the estimation of two-way interactions. Applications include gene-set testing \citep{chen2010two}, change-point detection \citep{wang2022inference} and kernel-based tests in machine learning \citep{gretton2012kernel}.

\vspace{.5em}

Classically, with $d(n)=d$ fixed, it is well-known that the distributions of both statistics admit either a Gaussian limit or a sum-of-chi-squares limit (up to reweighting and centering of the sum), depending on a degeneracy condition \citep{serfling1980approximation}. Recently, the high-dimensional case, where $d(n) \rightarrow \infty$ as $n \rightarrow \infty$, has gained attention due to its relevance in machine learning applications. Recent work investigates (variants of) U-statistics that admit a Gaussian limit despite degeneracy \citep{chen2010two,yan2021kernel,gao2023dimension,kim2020dimension,shekhar2022permutation}, or a sum-of-chi-squares limit despite non-degeneracy \citep{huang2023high,bhattacharya2022asymptotic}. In all cases, the error of Gaussian approximation for both U- and V-statistics is known to be $O(n^{-1/2})$ \citep{chen2011normal}, and is not improvable \citep{bentkus1994lower}. The behavior of non-Gaussian approximation is more nuanced: The known upper bounds depend on the number of non-zero eigenvalues of the Hilbert-Schmidt operator associated with the kernel $u$ of the U-statistic, ranging from $O(n^{-1})$ for five non-zero eigenvalues \cite{gotze2014explicit}, $O(n^{-1/12})$ for one non-zero eigenvalue \cite{yanushkevichiene2012bounds} (and control of an $18/5$th moment), to $O(n^{-1/14})$ without eigenvalue assumptions \cite{huang2023high}. This note constructs a set of i.i.d.~random vectors $X=(X_i)_{i \leq n}$ in $\R^{d(n)}$, whose $\nu$-th moments exist for $\nu \in (2,3]$, and functions $k_u$ and $k_v$, such that the following error bounds hold for the approximation by a quadratic form of Gaussians.

\begin{theorem} \label{thm:main} Fix $\nu \in (2,3]$. Let $\chi^2_1$ be a chi-squared random variable with $1$ degree of freedom, $\xi \sim \cN(0,1)$ be independent of $\chi^2_1$ and $\overline{\chi^2_1} = \chi^2_1 - 1$. There exist some constants $c, C > 0$, $N \in \N$ and a sequence $(\sigma_n)_{n \in \N}$ with $\sigma_n \rightarrow 0$, where $X$, $k_u$ and $k_v$ depend on $\sigma_n$, such that for all $n > N$ and $d(n) \in \N$,
\begin{align*}
    c n^{-\frac{\nu-2}{4\nu}} 
    \;\leq&\;
    \msup_{t \in \R} \Big| \P\big( \sqrt{n(n-1)} \, u_n(X) \leq t \big) - \P\Big( \, \sigma_n \xi + \overline{\chi_1^2}  \,\leq\, t \, \Big) \Big|
    \;\leq\;
    C n^{-\frac{\nu-2}{4\nu}}\;,
    \\
    c n^{-\frac{\nu-2}{4\nu }} 
    \;\leq&\;
    \msup_{t \in \R} \Big| \P\big( n \, v_n(X) \leq t \big) - \P\big( \, \sigma_n \xi + \chi^2_1 \leq t \, \big) \Big|
    \;\leq\;
    C n^{-\frac{\nu-2}{4\nu}}\;.
\end{align*}
\end{theorem}

\begin{remark} (i) When $\nu=3$, \cref{thm:main} says that the approximation error is $\Theta(n^{-\frac{1}{12}})$. (ii) In the construction, we choose $\sigma_n$ to decay as $\Theta\big(n^{-\frac{\nu-2}{2\nu}}\big)$. The approximation becomes a chi-squared approximation in the limit $n \rightarrow \infty$, but at a very slow rate. (iii) Since $\xi$ can be obtained as the limiting distribution of a partial sum of weighted and centred chi-squared variables, \cref{thm:main} can be read as a result on the approximation of $u_n(X)$ and $v_n(X)$ by infinite sums of weighted and shifted chi-squares. For further discussions on when a Gaussian component may emerge from an infinite sum of weighted chi-squares, see \cite{bhattacharya2022asymptotic}.
\end{remark}

\cref{thm:main} is an instance of the broader theme of Gaussian universality: Consider a set of i.i.d.~random vectors $Y = (Y_i)_{i \leq n}$ in $\R^b$ and the collection of Gaussian vectors
\begin{align*}
    Z \;=&\; (Z_1, \ldots, Z_n)
    &\text{ where }&&
    Z_i \;\overset{\rm i.i.d.}{\sim}&\; \cN(\mean[Y_1], \Var[Y_1])\;.
\end{align*}
The term Gaussian universality refers to the fact that, for a large class of functions $f: (\R^b)^n \rightarrow \R$, $f(Y)$ is close to $f(Z)$ in distribution as $n \rightarrow \infty$. This phenomenon can be observed even when the law of $Y_1$ is allowed to vary in $n$, and has found applications in probability, statistics and machine learning \citep{tao2011random,tao2015random,bayati2015universality,gerace2024gaussian,montanari2022universality,pesce2023gaussian}. Indeed, a crucial intermediate step in \cref{thm:main} is to write $u_n(X) = \tilde u_n(Y)$ and $v_n(X) = \tilde v_n(Y)$ for some intermediate functions $\tilde u_n$ and $\tilde v_n$, and approximate them by $\tilde u_n(Z)$ and $\tilde v_n(Z)$ respectively. The lower bound of \cref{thm:main} adapts a degree-$m$ polynomial $p^*_m(X)$ constructed in our recent work \cite{huang2024gaussian}---inspired by Senatov \cite{senatov1998normal}---which has a slow universality approximation error on the order $\Omega( n^{-\frac{\nu-2}{2\nu m}} )$. The upper bound of \cref{thm:main} improves upon our generic universality result in \cite{huang2024gaussian} by using an argument specific to $p^*_m(X)$, instead of applying Lindeberg's technique. To extend these results to $u_n$ and $v_n$, our proof also exploits a technique called variance domination in \cite{huang2024gaussian}, which is summarised in \cref{prop:VD} in \cref{sec:variance:domination}. 

\vspace{.5em}

An open question from \cite{huang2024gaussian} was whether the slow $n^{-1/12}$ universality approximation error holds also for degree-two U-statistics and V-statistics, which see more real-life applications. Yanushkevichiene
\cite{yanushkevichiene2012bounds} conjectures that the $n^{-1/12}$ rate is unimprovable for degree-two U-statistics in view of the construction by Senatov \cite{senatov1998normal}. \cref{thm:main} answers this question in the affirmative. An implication is that, without additional structural assumptions on the data distribution or the kernel functions $k_u$ and $k_v$, the slow $n^{-1/12}$ rate of quadratic-form-of-Gaussian approximation for U-statistics and V-statistics is not improvable. In \cref{sec:construction}, we provide the explicit constructions of $k_u$, $k_v$ and $X$. All proofs are included in \cref{sec:proofs}.

\section{Construction of $k_u$, $k_v$ and $X$} \label{sec:construction}

We first recall the lower bound construction from \cite{huang2024gaussian} in the case $m=2$. The construction involves a polynomial $p^*_n: (\R^2)^n \rightarrow \R$ given by 
\begin{align*}
    p^*_n(y_1,\ldots, y_n) 
    \;\coloneqq&\; 
    \mfrac{1}{\sqrt{n}} \msum_{i=1}^n y_{i1}  
    +  
    \Big(\mfrac{1}{\sqrt{n}}
    \msum_{i=1}^n y_{i2}\Big)^2
    \qquad
    \text{ for } y_i = (y_{i1}, y_{i2}) \in \R^2\;.
\end{align*} 
The collection of $\R^2$ random vectors $Y=(Y_1, \ldots, Y_n)$ are generated as follows. For a fixed $\sigma_0 > 0$ and $\nu \in (2,3]$, write $\sigma_n = \min\{\sigma_0 \, n^{-(\nu-2)/2\nu} \,,\, 1\}$. Note that this is the sequence $\sigma_n \rightarrow 0$ in \cref{thm:main}. Let $U_{\sigma_n}$ be the discrete random variable supported at three points with 
\begin{align*}
    U_{\sigma_n}
    \;=\; 
    \begin{cases}
        - 6^{-1/2} \sigma_n^{-2/(\nu-2)} \qquad &\text{ with probability } 2 \sigma_n^{2 \nu / (\nu-2)} \;,
        \\
        0 \qquad &\text{ with probability }  1 - 3 \sigma_n^{2 \nu / (\nu-2)} \;,
        \\
        2 \times 6^{-1/2} \sigma_n^{-2/(\nu-2)}  \qquad &\text{ with probability }  \sigma_n^{2 \nu / (\nu-2)} \;.
    \end{cases}
\end{align*}
$U_{\sigma_n}$ is constructed such that it becomes increasingly heavy-tailed as $\sigma_n \rightarrow 0$.  Now let $ U_{1;\sigma_n}, \ldots, U_{n;\sigma_n}$ be i.i.d.~copies of $U_{\sigma_n}$. The i.i.d.~random vectors $Y=(Y_i)_{i \leq n}$ are generated by
\begin{align*}
    Y_{i;\sigma_n}
    \;\coloneqq\; 
    \Big( \mfrac{1}{\sqrt{2}} U_{i;\sigma_n} + \mfrac{\sigma_n}{\sqrt{2}}\xi_{i1} \,,\, \xi_{i2} \Big) 
    \hspace{3em}
    \;\text{ where }\;\;
    \xi_{11}, \xi_{12}, \ldots, \xi_{n1}, \xi_{n2} 
    \;\overset{\rm i.i.d.}{\sim}\; 
    \cN(0,1)\;.
\end{align*}
To adapt $p^*_n(Y)$ to our setup, we first observe that $p^*_n$ can be rewritten as a V-statistic:
\begin{align*}
    p^*_n(y_1, \ldots, y_n) 
    \;=&\; 
    \mfrac{1}{n}
    \msum_{i,j=1}^n 
    \Big( 
        \mfrac{y_{i1}}{2\sqrt{n}}
        +
        \mfrac{y_{j1}}{2\sqrt{n}}
        +
        y_{i2} \, y_{j2}
    \Big)
    \;=\; 
    n \,
    \tilde v_n(y_1, \ldots, y_n)
    \;,
\end{align*}
where we have defined, for $y_1, \ldots, y_n \in \R^2$ and $a_1, a_2, b_1, b_2 \in \R$,
\begin{align*}
    \tilde v_n(y_1, \ldots,y_n)
    \;\coloneqq\;
    \mfrac{1}{n^2} \msum_{i,j=1}^n \tilde k_v(y_i, y_j)
    \;, 
    \hspace{1.5em}
    \tilde k_v( (a_1, a_2), (b_1, b_2) ) 
    \;\coloneqq\; 
    \mfrac{a_1}{2\sqrt{n}}
    +
    \mfrac{b_1}{2\sqrt{n}}
    +
    a_2 b_2
    \;.
\end{align*}
Moreover, since we have no restrictions on how $d(n)$ depends on $n$, there are non-unique choices of a function $\phi_{d(n)}: \R^{d(n)} \rightarrow \R$ and a probability measure $\mu_{d(n)}$ on $\R^{d(n)}$ such that 
\begin{align*}
    X_1 \;\sim&\; \mu_{d(n)} 
    &\Leftrightarrow&&
    \phi_{d(n)}(X_1) \;\overset{d}{=}&\; Y_{1; \sigma_n}\;.
\end{align*}
Our construction of the V-statistic is thus given by taking $X_i \overset{\rm i.i.d.}{\sim} \mu_{d(n)}$ and $k_v(x_1,x_2) \coloneqq \tilde k_v(\phi_{d(n)}(x_1), \phi_{d(n)}(x_2))$, which gives 
\begin{align*}
    v_n(X) \;=\; \mfrac{1}{n^2} \msum_{1 \leq i,j \leq n} k_v(X_i, X_j) \;\overset{d}{=} 
    \tilde v_n(Y)
    \;=\; 
    \mfrac{1}{n} \, p^*_n(Y)\;,
\end{align*}
where $\overset{d}{=}$ denotes equality in distribution. This makes the lower bound from \citep{huang2024gaussian} immediately applicable. For the U-statistics construction, we observe that
\begin{align*}
    p^*_n(y_1, \ldots, y_n) 
    \;=&\; 
    \mfrac{1}{\sqrt{n(n-1)}} 
    \msum_{i \neq j}^n 
    \Big( 
        \mfrac{y_{i1}}{2\sqrt{n-1}}
        +
        \mfrac{y_{i2}}{2\sqrt{n-1}}
        +
        \mfrac{\sqrt{n-1}}{\sqrt{n}} \, y_{i2} y_{j2}
    \Big)
    +
    \mfrac{1}{n} \msum_{i=1}^n y_{i2}^2
    \\
    \;=&\;
    \sqrt{n(n-1)} \, \tilde u_n(y_1, \ldots, y_n) 
    + 
    R_n(y_1, \ldots, y_n)
    \;,
\end{align*}
where we have defined, for $y_1, \ldots, y_n \in \R^2$ and $a_1, a_2, b_1, b_2 \in \R$,
\begin{align*}
    &\tilde u_n(y_1, \ldots, y_n) 
    \;\coloneqq\; 
    \mfrac{1}{n(n-1)} \msum_{i \neq j} \tilde k_u(y_i, y_j)
    \;,
    \hspace{1.5em}
    R_n(y_1, \ldots, y_n) 
    \;\coloneqq\;
    \mfrac{1}{n} \msum_{i=1}^n y_{i2}^2\;,
    \\
    &\tilde k_u( (a_1, a_2), (b_1, b_2) ) 
    \;\coloneqq\; 
    \mfrac{a_1}{2\sqrt{n-1}}
    +
    \mfrac{b_1}{2\sqrt{n-1}}
    +
    \mfrac{\sqrt{n-1}}{\sqrt{n}} \, a_2 b_2\;.
\end{align*}
Therefore, our construction of the U-statistic is to take $k_u(x_1,x_2) \coloneqq \tilde k_u(\phi_{d(n)}(x_1), \phi_{d(n)}(x_2))$, which gives 
\begin{align*}
    u_n(X) 
    \;=\; 
    \mfrac{1}{n(n-1)} \msum_{1 \leq i \neq j \leq n} k_u(X_i, X_j) 
    \;\overset{d}{=} 
    \tilde u_n(Y)
    \;=\; 
    p^*_n(Y) - R_n(Y)\;.
\end{align*}
The main technical task is thus to show that $p^*_n(Y)$ approximates a chi-squared distribution and that $R_n(Y)$ has negligible effect other than centering the chi-squared distribution.

\section{Variance domination} \label{sec:variance:domination}

One main technique used in the proofs is the idea of variance domination: When computing the limit of a sum of two dependent, real-valued random variables $X'+Y'$, it suffices to ignore $Y'$ in the limit provided that the variance of $Y'$ is negligible compared to that of $X'$. The same idea is used in \cite{huang2024gaussian} for extending their universality results to non-polynomial functions. The next result summarises the technique.

\begin{proposition} \label{prop:VD} Let $X'$ and $Y'$ be two $\R$-valued, possibly dependent random variables with $\mean[Y'] = 0$. Then for every $t \in \R$, 
\begin{align*}
    \big| \P( X'+Y' \leq t )  - &\P( X' \leq t ) \big|
    \\
    &\;\leq\;
    \inf_{\epsilon > 0}
    \Big(
    \max\big\{ \;
        \P( X' \in (t-\epsilon, t]) 
    \,,\, 
        \P( X' \in (t, t+\epsilon]) 
    \big\}
    + 
    \mfrac{\Var[Y']}{ \epsilon^2}
    \Big)\;.
\end{align*}
If we further have $\sigma^2_{X'} \coloneqq \Var[X'] > 0$, then 
\begin{align*}
    \big| \P\big(  \sigma^{-1}_{X'}  & (X'+Y') \leq t  \big)  
    - 
    \P\big( \sigma^{-1}_{X'} X' \leq t \big) \big|
    \\
     &\;\leq\;
    \inf_{\epsilon > 0}
    \Big(
    \max\big\{ \;
        \P( \sigma^{-1}_{X'}  X' \in (t-\epsilon, t]) 
    \,,\, 
        \P( \sigma^{-1}_{X'}  X' \in (t, t+\epsilon]) 
    \big\}
    + 
    \mfrac{\Var[Y']}{\Var[X'] \, \epsilon^2}
    \Big)\;.
\end{align*}
\end{proposition}

\begin{remark}  Several adaptations are useful: (i)  To swap the roles of $X'+Y'$ and $X'$, one may replace $X'$ and $Y'$ above by $X'+Y'$ and $-Y'$ respectively; (ii) To replace $\sigma_{X'}$ by another normalisation, e.g.~$\sqrt{\Var[X'+Y']}$, one may rescale $t$ and $\epsilon$ simultaneously.
\end{remark}

\cref{prop:VD} formalises the variance domination effect: If $\Var[Y']/\Var[X'] = o(1)$, by choosing $\epsilon=(\Var[Y']/\Var[X'])^{1/3}$, \cref{prop:VD} implies that the c.d.f.~difference $\big| \P(  \sigma^{-1}_{X'}  (X'+Y') \leq t  )  -  \P( \sigma^{-1}_{X'} X' \leq t ) \big| = o(1)$. Meanwhile, to get a finite-sample control, one needs an anti-concentration bound on $\sigma_{X'}^{-1} X$. By a triangle inequality, it also suffices if $\sigma_{X'}^{-1} X$ is well-approximated in distribution by some random variable $Z'$ with an anti-concentration bound. For polynomials of a random vector following a log-concave probability measure, a celebrated anti-concentration result is available due to \cite{carbery2001distributional}. This in particular applies to polynomials of Gaussian random vectors.

\begin{fact}[Carbery-Wright inequality, Theorem 8 of \cite{carbery2001distributional}] \label{lem:carbery:wright} Let $q_m(\bx)$ be a degree-$m$ polynomial of $\bx \in \R^d$ taking values in $\R$, and $\eta$ be an $\R^d$-valued random vector following a log-concave probability measure. Then there exists a constant $C$ independent of $q_m$, $d$, $m$ or $\eta$ such that, for every $\epsilon > 0$,
    \begin{align*}
        \P \big( | q_m(\eta) | \leq \epsilon \big) 
        \;\leq\; 
        C  m \, \epsilon^{1/m} (\mean[|q_m(\eta)|^2])^{-1/{2m}}
        \;.
    \end{align*}
\end{fact}

This immediately implies the following corollary of \cref{prop:VD}:

\begin{corollary} \label{cor:VD} Let $X'$, $Y'$ and $\sigma_{X'} > 0$ be defined as in \cref{prop:VD}, and $q_m(\eta)$ be given as in \cref{lem:carbery:wright}. Suppose $\Var[q_m(\eta)] = \Var\big[ \sigma^{-1}_{X'} X' \big] = 1$. Then there is an absolute constant $C>0$ such that for every $t \in \R$,
\begin{align*}
    \msup_{t \in \R}
    \big| 
    \P\big( X'+Y' & \leq t \big) - \P\big(  X' \leq t \big)
    \big|
    \\
    &
    \;\leq\;
    C m 
    \Big(\mfrac{\Var[Y']}{\Var[X']}\Big)^{\frac{1}{2m+1}}
    +
    2 \msup_{t \in \R} \big| 
        \P\big( \sigma^{-1}_{X'} X' \leq t \big) 
        - 
        \P\big( q_m(\eta) \leq t \big) \big|
    \;.
\end{align*}
\end{corollary}

\begin{proof}[Proof of \cref{cor:VD}] The result follows from combining the second statement of \cref{prop:VD} and \cref{lem:carbery:wright}, noting that $\mean[|q_m(\eta)|^2] \geq \Var[q_m(\eta)] = 1$, choosing $\epsilon = (\Var[Y'] / \Var[X'])^{m/(2m+1)}$ and taking a supremum over $t \in \R$.
\end{proof}

\begin{remark} The generality of \cref{prop:VD} comes at a cost: It does not provide the tightest control on the c.d.f.~difference in general, as the variance ratio term comes from the Markov's inequality. When one has more information about the tail behaviour of $Y'/\sqrt{\Var[X']}$, the bound can usually be improved.
\end{remark}

\section{Proofs} \label{sec:proofs}

Throughout this section, we denote the collection of Gaussian vectors
\begin{align*}
    Z \;=&\; (Z_1, \ldots, Z_n)
    &\text{ where }&&
    Z_i \;\overset{\rm i.i.d.}{\sim}&\;
    \cN(\mean[Y_{1;\sigma_n}], \Var[Y_{1;\sigma_n}])
    \;=\;
    \cN\Big( \begin{psmallmatrix}0 \\ 0 \end{psmallmatrix}, \begin{psmallmatrix} \sigma^2_n & 0 \\ 0 & 1 \end{psmallmatrix}\Big)
    \;.
\end{align*}
We state three intermediate results used in the proof of \cref{thm:main}. The first two results concerns the Gaussian universality approximations of $\tilde v_n(Y)$ and $\tilde u_n(Y)$. We improve the upper bound of \cite{huang2024gaussian}, $n^{-\frac{\nu-2}{4\nu+2}}$, by using an argument specific to the construction instead of the generic Lindeberg's technique used in \cite{huang2024gaussian}. Note that the $\sigma_0$-dependence below arises because $Y$ and $Z$ are implicitly parameterised by $\sigma_0$.

\begin{lemma} \label{lem:universality:v} Fix $\nu \in (2,3]$. Then there exist some absolute constants ${c, C, \sigma_0 > 0}$ and ${N \in \N}$, such that for all $n \geq N$,
    \begin{align*}
        c n^{-\frac{\nu-2}{4\nu }} 
        \;\leq&\; 
        \msup_{t \in \R} \big| \P( n \, \tilde v_n(Y) \leq t) - \P( n \, \tilde v_n(Z) \leq t) \big|
        \;\leq\;
        C n^{-\frac{\nu-2}{4\nu }} \;.
    \end{align*}
\end{lemma}

\begin{lemma} \label{lem:universality:u} Fix $\nu \in (2,3]$ and let $\sigma_0$ be given as in \cref{lem:universality:v}. Then there exist some absolute constants ${c', C' > 0}$ and ${N' \in \N}$, such that for all $n \geq N'$,
    \begin{align*}
        c' n^{-\frac{\nu-2}{4\nu }} 
        \;\leq&\; 
        \sup_{t \in \R} \Big| \P\Big( \sqrt{n(n-1)} \, \tilde u_n(Y) \leq t\Big) - \P\Big( \sqrt{n(n-1)} \, \tilde u_n(Z) \leq t\Big) \Big|
        \;\leq\;
        C' n^{-\frac{\nu-2}{4\nu }} \;.
    \end{align*}
\end{lemma}

Now observe that the universality approximations can be expressed as
\begin{align*}
    n \, \tilde v_n(Z) 
    \;=&\; 
    \mfrac{1}{\sqrt{n}}
    \msum_{i=1}^n 
    Z_{i1} 
    +
    \mfrac{1}{n} \msum_{i,j=1}^n Z_{i2} Z_{j2} 
    \;\overset{d}{=}\; 
    \sigma_n \xi + \chi^2_1
    \;,
    \\
    \sqrt{n(n-1)} 
    \, \tilde u_n(Z)
    \;=&\;
    \mfrac{1}{\sqrt{n}} 
    \msum_{i=1}^n 
    Z_{i1}
    +
    \mfrac{1}{n} 
    \msum_{i \neq j} 
    Z_{i2} Z_{j2}
    \;\overset{d}{=}\; 
    \sigma_n \xi 
    + 
    \mfrac{1}{n} 
    \msum_{i \neq j} 
    Z_{i2} Z_{j2}
    \;.
\end{align*}
The next result allows the quadratic part of $\sqrt{n(n-1)} 
\, \tilde u_n(Z)$ to be approximated by the centred chi-squared variable $\overline{\chi^2_1}$ plus a small Gaussian component $\sigma_n \xi$.

\begin{lemma} \label{lem:chi:square} There exists some absolute constant $C'' > 0$ such that for all $n \in \N$,  
    \begin{align*}
        \msup_{t \in \R}
        \Big| 
            \P\Big(  \sqrt{n(n-1)}\, \tilde u_n(Z)  \leq t \Big) 
            -
            \P\Big( \, \sigma_n \xi + \overline{\chi^2_1} \leq t \Big)
        \Big|
        \;\leq&\;
        C'' n^{-1/5}
        \;.
    \end{align*}
\end{lemma}

\subsection{Proof of \cref{thm:main}} Recall that by construction, $u_n(X) \overset{d}{=} \tilde u_n(Y)$, $v_n(X) \overset{d}{=} \tilde v_n(Y)$ and $n \tilde v_n(Z) \overset{d}{=} \sigma_n \xi + \chi^2_1$. The V-statistic bound then follows directly from \cref{lem:universality:v}: There exist some constants $\sigma_0, c_1, C_1 > 0$ and $N_1 \in \N$ such that for all $n \geq N_1$,
\begin{align*}
    c_1 n^{-\frac{\nu-2}{4\nu }} 
    \;\leq&\;
    \msup_{t \in \R} \Big| \P\big( n \, v_n(X) \leq t \big) - \P\big( \, \sigma_n \xi + \chi^2_1 \leq t \, \big) \Big|
    \;\leq\;
    C_1 n^{-\frac{\nu-2}{4\nu }}\;.
\end{align*}
By using the triangle inequality to combine the U-statistic bounds from \cref{lem:universality:u} and \cref{lem:chi:square}, there is another absolute constant $C_2 > 0$ such that  
\begin{align*}
    \sup_{t \in \R} 
    \Big| 
        \P\big( \sqrt{n(n-1)} \, u_n(X) \leq t \big) -
        \P\Big( \, \sigma_n \xi + \overline{\chi^2_1} \leq t \Big)
    \Big|
    \;\leq&\; 
    C_1 n^{-\frac{\nu-2}{4\nu}} + C_2 n^{-\frac{1}{5}}
    \;,
    \\
    \sup_{t \in \R} 
    \Big| 
        \P\big( \sqrt{n(n-1)} \, u_n(X) \leq t \big) -
        \P\Big( \, \sigma_n \xi + \overline{\chi^2_1} \leq t \Big)
    \Big|
    \;\geq&\; 
    c_1 n^{-\frac{\nu-2}{4\nu}} - C_2 n^{-\frac{1}{5}}\;.
\end{align*}
Since $\frac{\nu-2}{4\nu} \leq \frac{1}{8} < \frac{1}{5}$ for $\nu \in (2,3]$, the $n^{-1/5}$ term can be ignored when $n$ is large. In particular, there exist $c \in (0, c_1]$, $C > C_1$ and an integer $N \geq N_1$ such that for all $n \geq N$,
\begin{equation*}
    c n^{-\frac{\nu-2}{4\nu }} 
    \;\leq\;
    \msup_{t \in \R} \Big| \P\big( n \, v_n(X) \leq t \big) - \P\big( \, \sigma_n \xi + \chi^2_1 \leq t \, \big) \Big|
    \;\leq\;
    C n^{-\frac{\nu-2}{4\nu}}\;.
\end{equation*}
\qed

\subsection{Proof of \cref{prop:VD}} The proof relies on the following result: 

\begin{lemma}[Lemma 25 of \cite{huang2024gaussian}] \label{lem:approx:XplusY:by:Y} Let $X'$ and $Y'$ be two real-valued random variables. For any $a, b \in \R$ and $\epsilon > 0$, we have
    \begin{align*}
        \P( a \leq X'+Y' \leq b) 
        \;\leq\;
        \P( a - \epsilon \leq X' \leq b+\epsilon ) + \P(  |Y'| \geq \epsilon )\;,
        \\
        \P( a \leq X'+Y' \leq b) 
        \;\geq\;
        \P( a + \epsilon \leq X' \leq b-\epsilon ) - \P(  |Y'| \geq \epsilon )
        \;.
    \end{align*}
\end{lemma}

\begin{proof}[Proof of \cref{prop:VD}] By \cref{lem:approx:XplusY:by:Y} followed by the Markov's inequality, we get that for every $t \in \R$ and $\epsilon > 0$,
\begin{align*}
    &
    \;
    \P( X'+Y' \leq t) 
    \;\leq\;
    \P( X' \leq t+\epsilon ) + \P(  |Y'| \geq \epsilon )
    \;\leq\;
    \P( X' \leq t+\epsilon ) + \mfrac{\Var[Y']}{\epsilon^2}
    \;,
    \\
    &\;
    \P( X'+Y' \leq t) 
    \;\geq\;
    \P( X' \leq t-\epsilon ) - \mfrac{\Var[Y']}{\epsilon^2}
    \;.
\end{align*}
Subtracting $\P(X' \leq t)$ from both sides, we get that 
\begin{align*}
    \big| \P( X'+Y' \leq t)  - \P(X' \leq t) \big| 
    \;\leq\; 
     \max\big\{ \;
        \P( X' \in (t-\epsilon, t]) 
    \,,\, 
        \P( X' \in (t, t+\epsilon]) 
    \big\} + \mfrac{\Var[Y']}{\epsilon^2}\;,
\end{align*}
and taking an infimum over $\epsilon > 0$ gives the first bound. The second bound follows by rescaling $t$ and $\epsilon$ at the same time by $\sigma_X' > 0$.
\end{proof}

\subsection{Proof of \cref{lem:universality:v}} The lower bound follows directly from Theorem 2 of \cite{huang2024gaussian} by noting that $n \tilde v_n = p^*_n$, so it suffices to prove the upper bound. Denote 
\begin{align*}
    S_n \;\coloneqq\; \mfrac{1}{\sqrt{n}} \msum_{i=1}^n \Big( \mfrac{1}{\sqrt{2}} U_{i;\sigma_n} + \mfrac{\sigma_n}{\sqrt{2}} \xi_{i1} \Big)
\end{align*}
such that, for $\xi \sim \cN(0,1)$ independent of all other variables,
\begin{align*}
    n \tilde v_n(Y) \;\overset{d}{=}&\; S_n + \chi^2_1 
    &\text{ and }&&
    n \tilde v_n(Z) \;\overset{d}{=}&\; \sigma_n \xi + \chi^2_1\;.
\end{align*}
By Lemmas 19 and 20 of \cite{huang2024gaussian}, under the choice of $\sigma_0$ and $\tilde N$ in Theorem 2 of \cite{huang2024gaussian}, we have that for all $n \geq \tilde N$,
\begin{align*}
    &\,
    \sup_{x \in \R} \Big| \P( S_n < x ) - \P( \sigma_n \xi < x ) - 
    \mfrac{A}{n^{1/2} \sigma_n^{\nu/(\nu-2)}} \Big( 1 - \mfrac{x^2}{\sigma_n^2} \Big) e^{-x^2/(2\sigma_n^2)} \Big|
    \;\leq\;
    \mfrac{2}{n \sigma^{2 \nu / (\nu-2)}} 
    \;,
    \\
    &\, 
    \big|  \P( S_n < x ) - \P( \sigma_n \xi < x ) \big| \;\leq\; B \bigg( \, \mfrac{1}{n^{1/2} \sigma_n^{\nu/(\nu-2)}} e^{-x^2 / 16 \sigma_n^2} + \mfrac{1}{n^{3/2} x^4 \sigma_n^{(8-\nu)/(\nu-2)}} \bigg)
\end{align*}
for some absolute constants $A,B > 0$. By splitting an integral and using these bounds, 
\begin{align*}
    &\;
    \big| \P( n \, \tilde v_n(Y) < t) - \P( n \, \tilde v_n(Z) < t) \big|
    \\
    &\;=\;
    \Big| \mint_{-\infty}^\infty \big( \P( S_n < t - y^2 ) -  \P( \sigma_n \xi < t - y^2 )  \big) \, \mfrac{e^{-y^2/2}}{\sqrt{2\pi}} dy \Big|
    \\
    &\;\leq\;
    \mint_{|t-y^2| <  \sigma_n} \big| \P( S_n < t - y^2 ) -  \P( \sigma_n \xi < t - y^2 )  \big| \, \mfrac{e^{-y^2/2}}{\sqrt{2\pi}} dy 
    \\
    &\qquad 
    +
    \mint_{|t-y^2| \geq  \sigma_n} \big| \P( S_n < t - y^2 ) -  \P( \sigma_n \xi < t - y^2 )  \big| \, \mfrac{e^{-y^2/2}}{\sqrt{2\pi}} dy 
    \\
    &\;\leq\;
    \mfrac{A}{\sqrt{2\pi}\, n^{1/2} \sigma_n^{\nu/(\nu-2)}}
    \mint_{|t-y^2| <  \sigma_n} \Big|  1 - \mfrac{(t-y^2)^2}{\sigma_n^2} \Big| \, e^{-\frac{(t-y^2)^2}{2\sigma_n^2}-\frac{y^2}{2}} dy 
    +
    \mfrac{4 \sqrt{ \sigma_n}}{n \sigma_n^{2\nu/(\nu-2)}}
    \\
    &\qquad 
    +
    \mfrac{B}{\sqrt{2\pi}\, n^{1/2} \sigma_n^{\nu/(\nu-2)}}
    \mint_{|t-y^2| \geq  \sigma_n} e^{-\frac{(t-y^2)^2}{16 \sigma_n^2} - \frac{y^2}{2} } dy 
    \\
    &\qquad
    +
    \mfrac{B}{\sqrt{2\pi} \, n^{3/2} \sigma_n^{(8-\nu)/(\nu-2)}} 
    \mint_{|t-y^2| \geq  \sigma_n} \mfrac{1}{(t-y^2)^4} \, e^{-\frac{y^2}{2}} dy
    \\
    &\;\leq\;
    \mfrac{2 A \sqrt{ \sigma_n}}{\sqrt{2\pi}\, n^{1/2} \sigma_n^{\nu/(\nu-2)}} 
    +
    \mfrac{4 \sqrt{ \sigma_n}}{n \sigma_n^{2\nu/(\nu-2)}}
    +
    \mfrac{B \sqrt{\sigma_n} }{\sqrt{2\pi}\, n^{1/2} \sigma_n^{\nu/(\nu-2)}}
    \mint_{|\sigma_n^{-1} t -y^2 | \geq 1} e^{-\frac{(\sigma_n^{-1} t-y^2)^2}{16}} dy 
    \\
    &\qquad 
    + \mfrac{B \sqrt{\sigma_n }}{\sqrt{2\pi}\,n^{3/2} \sigma_n^{(8-\nu)/(\nu-2)} \sigma_n^4} \mint_{|\sigma_n^{-1} t -y^2 | \geq 1} \, \mfrac{1}{(\sigma_n^{-1}  t - y^2)^4} \, dy 
    \;.
\end{align*}
In the last line, we have used a change-of-variable $y \mapsto \sqrt{\sigma_n} y$ and noted that $e^{-y^2/2} \leq 1$. Now fix $N \geq \tilde N$ such that $\sigma_n = \sigma_0 n^{-(\nu-2)/(2\nu)}$ for all $n \geq N$, which only depends on the absolute constant $\sigma_0 > 0$; in this case, 
\begin{align*}
    \mfrac{\sqrt{\sigma_n}}{n^{1/2} \sigma_n^{\nu/(\nu-2)}} \;=\; \sigma_0^{\frac{1}{2} - \frac{\nu}{\nu-2}} n^{-\frac{\nu-2}{4\nu}}
    \;,
    \qquad 
    \mfrac{\sqrt{\sigma_n}}{n^{3/2} \sigma_n^{(8-\nu)/(\nu-2)} \sigma_n^4} 
    \;=\;
    \sigma_0^{\frac{1}{2} - \frac{3\nu}{\nu-2}}
    n^{-\frac{\nu-2}{4\nu}}\,.
\end{align*}
Then there exists some constant $A'$ that depends only on $\sigma_0$ such that, for all $n > N$,
\begin{align*}
    &\;
    \big| \P( n \, \tilde v_n(Y) < t) - \P( n \, \tilde v_n(Z) < t) \big|
    \\
    &\;\leq\;
    A' n^{-\frac{\nu-2}{4\nu}}
    \Big( 1 
        +  \mint_{|\sigma_n^{-1} t -y^2 | \geq 1} e^{-\frac{(\sigma_n^{-1} t-y^2)^2}{16}} dy  
        +  \mint_{|\sigma_n^{-1} t -y^2 | \geq 1} \, \mfrac{1}{(\sigma_n^{-1}  t - y^2)^4} \, dy
    \Big)
    \\
    &\;\eqqcolon\;  A' n^{-\frac{\nu-2}{4\nu}} (1 + I_1(\sigma_n^{-1}t) + I_2(\sigma_n^{-1}t) )\;,
\end{align*}
where we write $I_1(\tau) \coloneqq \int_{|\tau - y^2| \geq 1} e^{-(\tau-y^2)^2/16} dy$ and $I_2(\tau) \coloneqq \int_{|\tau - y^2| \geq 1} (\tau-y^2)^{-4} dy$. To handle $I_1(\tau)$, we split the integral further and use a change-of-variable with $z=y^2$ to obtain 
\begin{align*}
    I_1(\tau) 
    \;\leq&\; \mint e^{-(\tau-y^2)^2/16} dy  
    \;=\; 
    \mint_{ y^2 < 1 } \, e^{-(\tau-y^2)^2/16}  dy  
    +
    2 \mint_{ y^2 \geq 1, y \geq 0 } \, e^{-(\tau-y^2)^2/16}  dy  
    \\
    \;\leq&\; 
    2 \, +\,  2 \mint_{z \geq 1}  e^{-(\tau-z)^2/16} \mfrac{1}{2\sqrt{z}} \, dz
    \\
    \;\leq&\;
    2 \, +\,  \mint_{z \geq 1}  e^{-(\tau-z)^2/16} \, dz
    \;\leq\;  2 \,+\, \sqrt{2 \pi \times 8}  \mint  \mfrac{e^{-(\tau-z)^2/16}}{\sqrt{2 \pi \times 8}} \, dz
    \;=\; 2 + 4\sqrt{\pi}\;.
\end{align*}
A similar strategy applied to $I_2(\tau)$ gives 
\begin{align*}
    I_2(\tau) 
    \;=&\;
    \mint_{|\tau - y^2| \geq 1, y^2 < 1} \, \mfrac{1}{(\tau-y^2)^4} \, dy 
    +
    2 \mint_{|\tau - y^2| \geq 1, y^2 \geq 1, y \geq 0} \, \mfrac{1}{(\tau-y^2)^4} \, dy 
    \\
    \;\leq&\;
    2 + 2 \mint_{|\tau - z| \geq 1, z \geq 1} \, \mfrac{1}{(\tau-z)^4} \mfrac{1}{2\sqrt{z}} \, dz
    \\
    \;\leq&\;
    2 + \mint_{|\tau - z| \geq 1} \, \mfrac{1}{(\tau-z)^4} \, dz
    \;=\;
    2 + \mint_{|z'| \geq 1} \, \mfrac{1}{(z')^4} \, d(z') \;=\; \mfrac{8}{3}\;.
\end{align*}
Substituting these two bounds back and noting that the resulted bounds do not depend on $t$, we get that there exist some constants $C > 0$ and $N \in \N$ that depend only on the absolute constant $\sigma_0 > 0$ such that for all $n \geq N$,
\begin{align*}
    \msup_{t \in \R} \big| \P( n \, \tilde v_n(Y) < t) - \P( n \, \tilde v_n(Z) < t) \big|
    \;\leq\; C n^{-\frac{\nu-2}{4\nu}} \;.
\end{align*}
\qed

\subsection{Proof of \cref{lem:universality:u}} Notice that
\begin{align*}
    \sqrt{n(n-1)} \tilde u_n \;=\; p^*_n - R_n \;=\; n \tilde v_n - R_n\;,
\end{align*}
and we already have the universality approximation bound for $n \tilde v_n$ from \cref{lem:universality:v}. It suffices to apply variance domination to approximate $\sqrt{n(n-1)} \tilde u_n$ by $n \tilde v_n$, which requires us to compute the relevant variances. Write $Z_i=(Z_{i1}, Z_{i2})$, where $Z_{i1}$ is the Gaussian component that match the first two moments of $U_{i;\sigma_n} + 2^{-1/2} \sigma_n \xi_{i1}$ and $Z_{i2}$ is the Gaussian component that matches $\xi_{i2} \sim \cN(0,1)$ in distribution. Then by independence, 
\begin{align*}
    \sigma_*^2 
    \;\coloneqq\; 
    \Var[p^*_n(Y)] 
    \;=&\;
    \Var[p^*_n(Z)] 
    \;=\; 
    \Var\Big[ \mfrac{1}{\sqrt{n}} \msum_{i=1}^n Z_{i1} \Big]
    +
    \Var\Big[ \Big( \mfrac{1}{\sqrt{n}} \msum_{i=1}^n Z_{i2}  \Big)^2 \Big]
    \\
    \;\geq&\;
    \Var\Big[ \Big( \mfrac{1}{\sqrt{n}} \msum_{i=1}^n Z_{i2}  \Big)^2 \Big]
    \;=\;
    \Var[ Z_{12}^2]
    \;=\;
    2
    \;.
\end{align*}
By independence again, 
\begin{align*}
    \Var[ R_n(Y) ] 
    \;=\;
    \Var[ R_n(Z) ] 
    \;=\;
    \Var\Big[ \mfrac{1}{n} \msum_{i=1}^n Z_{i2}^2  \Big]
    \;=\;
    \mfrac{\Var[Z_{i2}^2 ]}{n} 
    \;=\;
    \mfrac{2}{n} 
    \;,
\end{align*}
and also note that 
\begin{align*}
    \mean[R_n(Y)] \;=\; \mean[R_n(Z)] \;=\; 1 \;.
\end{align*}
We now prove the upper bound by combining the variance domination result in \cref{cor:VD} and the upper bound of \cref{lem:universality:v}. Let $\sigma_0$ be given as in \cref{lem:universality:v}. Since 
\begin{align*}
    \sqrt{n(n-1)} \, \tilde u_n(Y) \;=\; (p^*_n(Y) - 1) - (R_n - \mean[R_n(Y)])\;,
\end{align*}
there are some absolute constants $\tilde C' > 0$ and $N \in \N$ such that for all $n \geq N$ and every $t \in \R$,
\begin{align*}
    \big| 
     \P\big( & \sqrt{n(n-1)} \,  \tilde u_n(Y)  \leq t \big) - \P\big( p^*_n(Y) - 1 \leq t \big)
    \big|
    \\
    &
    \;\leq\;
    \tilde C' 
    \Big(\mfrac{\Var[R_n(Y)]}{\Var[p^*_n(Y)]}\Big)^{\frac{1}{5}}
    +
    2 \msup_{\tau \in \R} \big| 
        \P\big( \sigma_*^{-1} (p^*_n(Y) - 1) \leq \tau \big) 
        - 
        \P\big( \sigma_*^{-1} (p^*_n(Z) - 1)\leq \tau \big) \big|
    \\
    &
    \;\leq\;
    \tilde C' n^{-\frac{1}{5}} + 2 C n^{-\frac{\nu-2}{4\nu}}
    \;\leq\; 
    \tilde C'_*  n^{-\frac{\nu-2}{4\nu}}
\end{align*}
for some absolute constants $C, \tilde C'_* > 0$; in the last line, we have used that $\frac{\nu-2}{4\nu} \leq \frac{1}{8} < \frac{1}{5}$ for $\nu \in (2,3]$. Similarly we have 
\begin{align*}
    \big| 
     \P\big( \sqrt{n(n-1)} \,  \tilde u_n(Z)  \leq t \big) - \P\big( (p^*_n(Z) - 1) \leq t \big)
    \big|
    \;\leq\;  \tilde C'_* n^{-\frac{\nu-2}{4\nu}}\;.
\end{align*}
Combining both bounds with the upper bound of \cref{lem:universality:v} by a triangle inequality, we get the desired upper bound that for some absolute constant $C' > 0$ and all $n \geq N$,
\begin{align*}
    \msup_{t \in \R} \Big| \P\Big( \sqrt{n(n-1)} \, \tilde u_n(Y) \leq t\Big) - \P\Big( \sqrt{n(n-1)} \, \tilde u_n(Z) \leq t\Big) \Big|
    \;\leq\;
    C' n^{-\frac{\nu-2}{4\nu }}
    \;.
\end{align*}
For the lower bound, we apply Lemma 21 of \cite{huang2024gaussian}: There exist some constants $c_* > 0$ and $\tilde N \in \N$ depending only on the fixed constant $\sigma_0$---which is chosen to be the one used in Theorem 2 of \cite{huang2024gaussian} and also \cref{lem:universality:v}---such that for any $n \geq \tilde N$,
\begin{align*}
    \P\big(  p^*_n(Z)< - 2\sigma_n \big) - \P\big(  p^*_n(Y) < - 2\sigma_n \big) \;\geq\; c_* \, n^{-\frac{\nu-2}{4\nu}}\;.
\end{align*}
To exploit this lower bound, we shall apply \cref{lem:approx:XplusY:by:Y} directly: For any $\epsilon > 0$,
\begin{align*}
    \P\big( &\sqrt{n(n-1)} \, \tilde u_n(Z) < - 2 \sigma_n - \sigma_* \epsilon - 1 \big)
    -
    \P\big( \sqrt{n(n-1)} \, \tilde u_n(Y) < - 2 \sigma_n - \sigma_* \epsilon - 1 \big)
    \\
    \;=&\;
    \P\big( \sqrt{n(n-1)} \, \tilde u_n(Z) < - 2 \sigma_n - \sigma_* \epsilon - 1 \big)
    -
    \P\big( p^*_n(Z) - 1 <  -2\sigma_n - 2 \sigma_* \epsilon - 1 \big)
    \\
    &\; 
    -
    \P\big( -2 \sigma_n - 2 \sigma_* \epsilon \leq p^*_n(Z) <  -2\sigma_n   \big)
    \\
    &\; 
    +
    \P\big(  p^*_n(Z)< - 2\sigma_n \big) - \P\big(  p^*_n(Y) < - 2\sigma_n \big) 
    \\
    &\;
    +
    \P\big(  p^*_n(Y) - 1 < - 2\sigma_n - 1 \big) 
    -
    \P\big( \sqrt{n(n-1)} \, \tilde u_n(Y) < - 2 \sigma_n - \sigma_* \epsilon - 1 \big)
    \\
    \;\geq&\;
    - 2 \P\big( \big| R_n(Y) - 1 \big| \geq \sigma_* \epsilon \big)
    - \P\big( | p^*_n(Z) + 2\sigma_n + \sigma_* \epsilon| \leq  \sigma_* \epsilon  \big)
    + c_* \, n^{-\frac{\nu-2}{4\nu}}
    \\
    \;\overset{(a)}{\geq}&\;
    - 
    \mfrac{2 \Var[R_n(Y)]}{\sigma^2_* \epsilon^2}
    - 
    \mfrac{\tilde c'_* (\sigma_* \epsilon)^{1/2}}{(\mean[ |p^*_n(Z) + 2\sigma_n + \sigma_* \epsilon|^2 ])^{1/4}}
    + 
    c_* \, n^{-\frac{\nu-2}{4\nu}}
    \\
    \;\overset{(b)}{\geq}&\; 
    -
    \mfrac{2}{\epsilon^2 n} 
    -
    \tilde c'_* \epsilon^{1/2}
    + 
    c_* \, n^{-\frac{\nu-2}{4\nu}}
\end{align*}
for some absolute constant $\tilde c'_* > 0$. In $(a)$, we have used the Markov's inequality and the Carbery-Wright inequality (\cref{lem:carbery:wright}); in $(b)$, we have plugged in the bounds on $\Var[R_n(Y)]$ and $\sigma_*$, and noted that $\mean[ |p^*_n(Z) + 2\sigma_n + \sigma_* \epsilon|^2 ] \geq \Var[p^*_n(Z)] = \sigma_*^2$. Taking $\epsilon = n^{-2/5}$ gives 
\begin{align*}
    \sup_{t \in \R} \Big| \P\Big( \sqrt{n(n-1)} \, \tilde u_n(Y) \leq t\Big) \,-\, \P\Big( \sqrt{n(n-1)} \, \tilde u_n(Z) \leq t\Big) \Big|
    \geq
    c_* \, n^{-\frac{\nu-2}{4\nu}} 
    - 
    \Big(\mfrac{4}{5} + \tilde c'_*\Big) n^{-\frac{1}{5}}\,,
\end{align*}
Since $\frac{\nu-2}{4\nu} \leq \frac{1}{8} < \frac{1}{5}$, there are some absolute constants $c'> 0$ and $N' \geq \tilde N$ such that the desired lower bound holds for all $n \geq N'$. \qed

\subsection{Proof of \cref{lem:chi:square}} Recall from the proof of \cref{lem:universality:u} that $\mean[R_n(Z)] = 1$, $\Var[R_n(Z)] = \frac{2}{n}$ and that $\Var[ n \tilde v_n(Z) ] = \Var[  p^*_n(Z) ]  \geq 2 $. Since
\begin{align*}
    \sqrt{n(n-1)} \, \tilde u_n(Z) 
    \;=\; 
    n \tilde v_n(Z) - \mfrac{1}{n} \msum_{i=1}^n Z_{i2}^2 
    \;=\;
    \big( n \tilde v_n(Z) - 1 \big) 
    - 
    \Big( \mfrac{1}{n} \msum_{i=1}^n Z_{i2}^2  - 1 \Big)\;,
\end{align*}
we can again apply variance domination (\cref{cor:VD}) to obtain that
\begin{align*}
    \sup_{t \in \R}
    \Big| 
        \P\Big(  \sqrt{n(n-1)}\, \tilde u_n(Z)  \leq t \Big) 
        -
        \P\Big(  n \tilde v_n(Z) - 1 \leq t \Big)
    \Big|
    \;\leq&\;
    C \bigg( \mfrac{\Var[ R_n(Z) ]}{ \Var[ n \tilde v_n(Z) - 1] } \bigg)^{\frac{1}{5}}
    \;=\;
    C n^{-\frac{1}{5}}
\end{align*}
for some absolute constant $C > 0$. Noting that $ n \tilde v_n(Z) - 1 \overset{d}{=} \sigma_n \xi + \overline{\chi^2_1}$ finishes the proof.
\qed

\subsection*{Acknowledgements}
This work has been supported by the Gatsby Charitable Foundation.

\newpage 
\bibliography{ref}

\end{document}